\documentclass[12pt]{amsart}
\usepackage[all]{xy}
\usepackage{a4wide}
\usepackage{amssymb}
\usepackage{amsthm}
\usepackage{amsmath}
\usepackage{amscd}
\usepackage{verbatim}
\usepackage[all]{xy}
\usepackage{framed,color}

\usepackage{graphicx}
\usepackage{subfig}

\usepackage{amsmath,amsthm,amssymb}
\usepackage{enumitem}

\usepackage{float}
\usepackage{ifthen}
\usepackage{xargs}
\usepackage{url}
\usepackage[all]{xy}
\usepackage{xcolor}
\usepackage{comment}
\usepackage[obeyDraft,textsize=tiny]{todonotes}
\specialcomment{draftnote}{\vspace{1em}\begingroup\ttfamily\footnotesize\color{blue}\begin{minipage}{15cm}}{\end{minipage}\endgroup\vspace{1em}}
\specialcomment{alert}{\vspace{1em}\begingroup\ttfamily\footnotesize\color{blue}\begin{minipage}{15cm}}{\end{minipage}\endgroup\vspace{1em}}

\usepackage[refpage,intoc,noprefix]{nomencl}

\newcommand{\R}{\mathbb{R}} 
\newcommand{\C}{\mathbb{C}} 
\newcommand{\Z}{\mathbb{Z}} 

\newcommand{\calE}{\mathcal{E}}

\def\({\left(}
\def\){\right)}

\def\lp{\left(}
\def\rp{\right)}

\newtheorem{theorem}{Theorem}[section]
\newtheorem{proposition}[theorem]{Proposition}
\newtheorem{lemma}[theorem]{Lemma}
\newtheorem{corollary}[theorem]{Corollary}
\newtheorem{question}[theorem]{Question}

\theoremstyle{definition}

\newtheorem{definition}[theorem]{Definition}

\newtheorem{remark}[theorem]{Remark}

\usepackage{comment}

\usepackage[obeyDraft,textsize=tiny]{todonotes}

\title{The Riemann Hypothesis for period polynomials of cusp forms}

\author{William Craig}
\address{Mathematical Institute,
Universit\"{a}t zu Köln}
\email{wcraig@uni-koeln.de}

\author{Wissam Raji}
\address{Department of Mathematics at the American University of Beirut (AUB), Founder of Number Theory unit at Center for Advanced Mathematical Sciences (CAMS)}
\email{wr07@aub.edu.lb}

\begin{document}
\maketitle

\begin{abstract}
    We consider the period polynomials $r_f(z)$ associated with cusp forms $f$ of weight $k$ on all of $\mathrm{SL}_2\lp \Z \rp$, which are generating functions for the critical $L$-values of the modular $L$-function associated to $f$. In 2014, El-Guindy and Raji proved that if $f$ is an eigenform, then $r_f(z)$ satisfies a ``Riemann hypothesis" in the sense that all its zeros lie on the natural boundary of its functional equation. We show that this phenomenon is not restricted to eigenforms, and we provide large natural infinite families of cusp forms whose period polynomials almost always satisfy the Riemann hypothesis. For example, we show that for weights $k \geq 120$, linear combinations of eigenforms with positive coefficients always have unimodular period polynomials.
\end{abstract}

\section{Introduction}

Let $S_k$ be the space of cusp forms having weight $k$ and level 1. In this paper, we study the period polynomial of $f$. To define the period polynomial of $f \in S_k$ with Fourier expansion $f(z) = \sum_{n \geq 1} a_n q^n$, $q := e^{2\pi i z}$, consider the {\it Eichler integral} of $f$, which is given by
\begin{align*}
    \calE_f(z):=\int_z^{i\infty}f(\tau)(\tau-z)^{k-2}\, d\tau=-\frac{(k-2)!}{(2\pi i)^{k-1}}\sum_{n=1}^\infty \frac{a_n}{n^{k-1}}e^{2\pi i nz}.
\end{align*}
The {\it period polynomial} of $f$ is the polynomial of degree (at most) $k-2$ defining the error to modularity of the Eichler integral $\mathcal E_f(z)$, i.e.
\begin{align*}
    r_f(z) := \mathcal E_f(z) - z^{k-2} \mathcal E_f\lp - \dfrac{1}{z} \rp.
\end{align*}
By considering the Fourier expansion of $f$, it is easy to deduce that $r_f(z)$ is a generating function for the critical $L$-values of the modular $L$-function $L(f,s)$. More specifically, we have
\begin{align*}
	r_f\lp z \rp = - \dfrac{\lp k-2 \rp!}{\lp 2\pi i \rp^{k-1}} \sum_{n=0}^{k-2} \dfrac{\lp 2\pi i z \rp^n}{n!} L\lp  f, k - n - 1 \rp=\sum_{n=1}^{k-2}i^{n+k-1}\Lambda(f,k-n-1) z^n,
\end{align*}
where $$\Lambda(f,s):=(2\pi)^{-s}\Gamma(s)L(f,s)$$ is the completed $L$-function satisfying the functional equation $\Lambda(f,s)=\epsilon(f)\Lambda(f,k-s)$ with $\epsilon(f)=\pm 1$.
The theory of period polynomials has important implications for the theory of modular forms and their $L$-functions. For example, Manin has shown \cite{Manin} that the critical $L$-values occuring in the coefficients of $r_f(z)$ satisfy the following rationality conditions.
\begin{theorem}
Let f be a normalized cuspidal Hecke eigenform with rational Fourier coefficients. Then there exist $\omega_{\pm}(f)\in \mathbb{R}$ such that $\Lambda(f,s)/\omega_{+}(f)$ and $\Lambda(f,r)/\omega_{-}(f)$ are rational for all $s,r$ such that $1\leq s,r\leq k-1$ and $s$ even, $r$ odd.
\end{theorem}

In this paper, we focus on the zeros of period polynomials. This study was initiated by Conrey, Farmer and Imamoglu \cite{conreyfarmerimamoglu}, who showed that the odd period polynomial $r_f^-(z)$ for $f \in S_k$ an eigenform have all its nontrivial zeros on the unit disk. Shortly thereafter, El-Guindy and the second author \cite{elguindyraji} showed that all the zeros of $r_f(z)$ are on the unit disk if $f \in S_k$ is an eigenform. Jin, Ma, Ono and Soundararajan \cite{jinmaonosound} proved the analogous result for newforms of level $N$ and even weights, and Liu, Park and Song \cite{liuparksong} considered the problem for newforms of any level, Nebentypus, and weight. Because $r_f(z)$ obeys a functional equation under the change of variables $z \mapsto 1/\sqrt{N}z$ for cusp forms of level $N$, these theorems became known as the ``Riemann hypothesis for period polynomials." We shall also use this terminology frequently. Many other authors have considered similar problems; see for example \cite{BRW, diamantisrolen1, diamantisrolen2, imkim} for various related problems.

All previous investigations into the Riemann hypothesis for period polynomials of which the authors are aware have focused on the case of eigenforms (or newforms in the case of higher levels). We prove theorems of this type for very broad classes of cusp forms. In this direction, our first main result is an elementary and completely explicit sufficient condition for the Riemann hypothesis for the period polynomial $r_f(z)$ of an arbitrary cusp form $f$. Since $r_{Af}(z) = A r_f(z)$ for any constant $A$, we phrase our results for normalized cusp forms for the sake of simplicity.

\begin{theorem} \label{Main Theorem}
    Let $N \geq 1$, $k \geq 12$ be integers, and let $f(z) = q^N + \sum_{n > N} c(n) q^n$ be a nonzero normalized cusp form. Let $C_f$ the minimal constant such that $\left| c(n) \right| \leq C_f \sigma_0(n) n^{\frac{k-1}{2}}$ for all $n \geq N$. Suppose further that
    \begin{align} \label{Key inequality 0}
        \dfrac{4 e^{2\pi} C_f}{\lp N+1 \rp^{k/4}} + \dfrac{4 e^{2\pi} C_f \lp 2\pi \rp^{\lfloor \frac{k}{4} \rfloor} \left[ \sqrt{k} \log\lp 2k \rp + 1 \right]}{\lp \lfloor \frac{k}{4} \rfloor \rp!} &+ \dfrac{e^{2\pi N} \lp 2\pi N \rp^{\lfloor \frac k4 \rfloor}}{N^{k-1} \lp \lfloor \frac k4 \rfloor \rp!} \\ \notag &< \dfrac{1}{N^{k-1}} \lp e^{-2\pi N} - \dfrac{\lp 2\pi N \rp^{k/2}}{\lp k/2 \rp!} e^{2\pi N} \rp
    \end{align}
    and that the criteria of Proposition \ref{H zeros theorem} are satisfied for $m = \frac{k}{2} - 1$. Then $r_f(z)$ satisfies the Riemann hypothesis.
\end{theorem}

\begin{remark}
        The existence of such a constant $C_f$ is not obvious {\it a priori}, but is a famous consequence of Deligne's proof of the Weil conjectures. We also note that although the Sturm basis for $S_k$ implies that $C_f$ can be arbitrarily large, a normalized cusp form $f$ cannot have arbitrarily small values of $C_f$.
\end{remark}

Although Theorem \ref{Main Theorem} is not a necessary condition for $r_f(z)$ to satisfy the Riemann hypothesis, this test gives immediate insight into the period polynomials of large families of cusp forms. We discuss two particular applications of Theorem \ref{Main Theorem}. The first two give natural families of cusp forms whose period polynomials satisfy the Riemann hypothesis for sufficiently large weights.

\begin{theorem} \label{Positive sums theorem}
    Let $k \geq 120$ be an even integer, and let $r$ be the dimension of the complex vector space $S_k$ of cusp forms. Suppose $f_1, \dots, f_r$ is the normalized eigenbasis of $S_k$, and suppose $c_1, \dots, c_r$ are non-negative real numbers, not all zero. Then if $f = c_1 f_1 + \dots + c_r f_r$, then $r_f(z)$ satisfies the Riemann hypothesis.
\end{theorem}

\begin{remark}
     Since $r_{Af}(z) = A r_f(z)$ for any constant $A$, the condition $c_j \geq 0$ in (1) could be replaced with the assumption that $c_j \in \C$ have identical arguments (apart from those $c_j$ equal to 0). In particular, the results also hold if we assume $c_j \leq 0$.
\end{remark}

While this first result restricts to the case of linear combinations with positive coefficients, certain very general theorems can still be proven for more general linear combinations of eigenforms which permit both positive and negative coefficients. This is made explicit in the following theorem.

\begin{theorem} \label{Constellations}
    Let ${\bf c} = (c_1, \dots, c_m)$ and ${\bf a} = (a_1, \dots, a_n)$ be vectors of non-negative real numbers such that $\sum_{\ell=1}^m c_\ell > \sum_{j=1}^n a_j$. For $k \geq 12$ an integer, let $f_1, \dots, f_r$ be the eigenforms and let $$\mathcal F_k({\bf c}, {\bf a}) := \{ c_1 f_{i_1} + \dots + c_m f_{i_m} - a_1 f_{j_1} - \dots - a_n f_{j_n} : 1 \leq i_1, \dots, i_m, j_1, \dots, j_n \leq r \text{ distinct} \}$$ and $$\mathcal F({\bf c}, {\bf a}) := \bigcup_{k \geq 0} \mathcal F_k({\bf c}, {\bf a}).$$ Then there is a constant $k_0$ depending on $\mathbf{c}, \mathbf{a}$ such that for all $k \geq k_0$, $r_f(z)$ satisfies the Riemann hypothesis for all $f \in \mathcal F_k\lp \mathbf{c}, \mathbf{a} \rp$. In particular, at most finitely many $f \in \mathcal F\lp \mathbf{c}, \mathbf{a} \rp$ have the property that $r_f(z)$ does not satisfy the Riemann hypothesis.
\end{theorem}

These results suggest interesting probabilistic interpretations. For example, Theorem \ref{Positive sums theorem} suggests that if one selects a cusp form ``at random" from $S_k$, there should be a positive probability that $r_f(z)$ should satisfy the Riemann hypothesis. We will discuss this interpretation in more detail in Section \ref{Probability}. Additionally, Theorem \ref{Constellations} suggests that a generic cusp form whose coefficients are small compared to the weight should satisfy the period polynomial Riemann hypothesis. This idea will also be made concrete in later results, specifically in Theorem \ref{U-bound theorem}.

The remainder of the paper is organized as follows. In Section \ref{Preliminaries}, we discuss preliminary results necessary for the main proofs, including unimodular polynomials, elementary properties of modular $L$-functions, and zeros of certain polynomials associated with exponential functions. In Section \ref{Main Proofs} we prove Theorems \ref{Main Theorem} and \ref{U-bound theorem}, which are the main technical results of the paper, and we derive Theorems \ref{Positive sums theorem} and \ref{Constellations} as corollaries. In Section \ref{Probability}, we give address statistical questions and prove that a ``randomly selected period polynomial" has a positive probability of satisfying the Riemann hypothesis. We end in Section \ref{Further directions} to discuss further possible directions for future research, in particular discussing families of cusp forms for which Theorem \ref{Main Theorem} is not applicable but which still often have unimodular period polynomials.

\section*{Acknowledgements}

The first author received funding from the European Research Council (ERC) under the European Union’s Horizon 2020 research and innovation programme (grant agreement No. 101001179).  The second author would like to thank the Center for Advanced Mathematical Sciences (CAMS) at the American University of Beirut (AUB) for the support and Kathrin Bringmann for the fruitful research visit to Cologne, Germany during which this project arose. The authors thank Kathrin Bringmann for helpful comments that improved the manuscript.

\section{Preliminary results} \label{Preliminaries}

\subsection{Unimodular polynomials}

In this section, we give a brief overview of unimodular polynomials and useful criteria for proving that certain polynomials are unimodular. To study unimodular polynomials, we need the notion of self-reciprocal polynomials. A {\it self-reciprocal polynomial} is any polynomial $P(z) \in \C[z]$ for which there is some complex number $\varepsilon$ of norm 1 such that
\begin{align*}
    P(z) = \varepsilon z^{\deg(P)} P\lp 1/z \rp.
\end{align*}
Cohn \cite{cohn} has shown that a polynomial $P(z) \in \C[z]$ is unimodular if and only if it is self-reciprocal and $P^\prime(z)$ has all its zeros inside the closed unit disk. Using this criterion and employing an argument of Lal\'{i}n and Smyth \cite{lalinsmyth}, El-Guindy and the second author proved the following useful result for producing unimodular polynomials.

\begin{proposition}[{\cite[Theorem 2.2]{elguindyraji}}] \label{Unimodularity Criterion}
	Let $h(z) \in \C[z]$ be a degree $n$ polynomial with all its zeros in the closed unit disk. For any $d \geq n$ and $\lambda \in \C$, consider the polynomial
	\begin{align*}
		P\lp \lambda; z \rp = z^{d-n} h(z) + \lambda z^n \overline{h}\lp 1/z \rp.
	\end{align*}
	Then for any $|\lambda| = 1$ such that $P\lp \lambda; z \rp \not = 0$, we have that $P\lp \lambda; z \rp$ is unimodular.
\end{proposition}
In our application, $P\lp \lambda; z \rp$ will be the period polynomial, and we will construct the polynomial $h(z)$ using the functional equation of the completed $L$-function $\Lambda\lp f,s \rp$.

\subsection{Period polynomials and $L$-functions of cusp forms}

We now define the key objects of this paper. We fix a cusp form $f \in S_k$. Let its Fourier expansion be given by
\begin{align*}
    f(z) = \sum_{n \geq N} c(n) q^n,
\end{align*}
where as usual $q = e^{2\pi i z}$ and where $N = N(f) \geq 1$ is the order of vanishing of $f$ at $i\infty$. Note that if $f$ is an eigenform, then a basic result of the theory of eigenforms is that $N = 1$.

Recall that the $L$-function associated to $f$ is given for $\mathrm{Re}\lp s \rp > 1$ by
\begin{align*}
	L\lp f, s \rp := \sum_{n \geq N} \dfrac{c(n)}{n^s}.
\end{align*}
The special values of the $L$-function are then encoded by the coefficients of the period polynomial, which we recall from the introduction is given by
\begin{align*}
	r_f\lp z \rp = - \dfrac{\lp k-2 \rp!}{\lp 2\pi i \rp^{k-1}} \sum_{n=0}^{k-2} \dfrac{\lp 2\pi i z \rp^n}{n!} L\lp  f, k - n - 1 \rp.
\end{align*}
Because our concern is only the norm of the zeros of $r_f(z)$, we normalize for convenience to a polynomial with real coefficients. Letting $w := k-2$ for the remainder of the paper, we consider the modified period polynomial
\begin{align*}
	p_f(z) := - \dfrac{\lp 2 \pi i \rp^{w+1}}{w!} r_f\lp z/i \rp = \sum_{n=0}^w L\lp f, w - n + 1 \rp \dfrac{\lp 2 \pi z \rp^n}{n!}.
\end{align*}
Thus, $r_f(z)$ is unimodular if and only if $p_f(z)$ is unimodular.

We now derive a criterion for proving the unimodularity of $p_f(z)$. The key fact is the well-known functional equation
\begin{align} \label{functional equation}
    \dfrac{\lp 2\pi \rp^{k-s}}{\Gamma\lp k-s \rp} L\lp f,s \rp = i^k \dfrac{\lp 2\pi \rp^s}{\Gamma(s)} L\lp f, k-s \rp.
\end{align}
We use this functional equation to relate the critical values of the $L$-function $L\lp f,s \rp$. More specifically, for any integer $0 \leq n \leq w$ the identity
\begin{align} \label{functional equation 2}
    L\lp f, w-n+1 \rp = i^k \dfrac{\lp 2\pi \rp^{w-n} n!}{\lp 2\pi \rp^n \lp w-n \rp!} L\lp f, n+1 \rp
\end{align}
is easily derived from \eqref{functional equation}. A short calculation using \eqref{functional equation 2} shows that $p_f(z)$ is self-reciprocal. More specifically, we have
\begin{align*}
    p_f(z) = i^k z^w p_f\lp \dfrac 1z \rp = i^k \sum_{n=0}^w L\lp f, w-n+1 \rp \dfrac{\lp 2\pi \rp^n}{n!} z^{w-n}.
\end{align*}
Note that $k$ must be even, as $f$ is a level 1 holomorphic modular form, so $i^k = \pm 1$. This equation motivates the idea of utilizing Proposition \ref{Unimodularity Criterion}. We now set up a notation to make this connection explicit. Letting $m := \frac{w}{2} = \frac{k}{2} - 1$ for the remainder of the paper, we define the polynomial
\begin{align} \label{q-def}
	q_f(z) := \sum_{n=0}^{m-1} L\lp f, w-n+1 \rp \dfrac{\lp 2 \pi \rp^n}{n!} z^{m-n} + \dfrac{1}{2} L\lp f, \dfrac{k}{2} \rp \dfrac{\lp 2\pi \rp^m}{m!}.
\end{align}
Observe that the functional equation of the $L$-function implies the crucial identity
\begin{align} \label{Set-up}
	i^k p_f(z) = z^m q_f(z) + i^k z^m q_f\lp 1/z \rp.
\end{align}
By applying Proposition \ref{Unimodularity Criterion} to \eqref{Set-up}, we obtain the following criteria for unimodularity for $r_f(z)$.

\begin{proposition} \label{q zeros}
    Let $f \in S_k$ be a cusp form. If the zeros of $q_f(z)$ are inside the closed unit disk, then the period polynomial $r_f(z)$ is unimodular.
\end{proposition}

The criterion of Proposition \ref{q zeros} for unimodularity is precisely what we will use to prove the main results of the paper. The rest of the paper is dedicated to the techniques required to prove that the zeros of $q_f(z)$ are inside the closed unit disk for a wide class of cusp forms $f$. The basic idea is to replace the $L$-functions with approximate numerical values, which requires knowledge about the size of the coefficients of the $L$-function.

\subsection{Estimates for $L$-functions}

By Proposition \ref{q zeros}, the question of unimodularity of period polynomials is reduced to the study of the polynomial $q_f(z)$. Before explaining the rigorous technical details, we give a heuristic which explains the approach. First, it is easy to see that for any constant $A$ and any $f \in S_k$, $r_{Af}(X) = A r_f(z)$. Therefore, we may assume without loss of generality that $f \in S_k$ is normalized, i.e. that $f(z) = q^N + O\lp q^{N+1} \rp$. Because most of the size of an $L$-function comes from its earliest terms, we have the naive approximation $L(f,s) \approx N^{-s}$. Using this approximation, we see that
\begin{align*}
    q_f(z) \approx \sum_{n=0}^m N^{-w+n-1} \dfrac{\lp 2 \pi \rp^n}{n!} z^{m-n} = N^{-w-1} \sum_{n=0}^m \dfrac{\lp 2 \pi N \rp^n}{n!} z^{m-n}.
\end{align*}
The polynomial appearing on the right-hand side is closely related to exponential functions. In particular, for any real $\beta > 0$ consider the standard Taylor series decomposition
\begin{align*}	
	\exp\lp 2\pi \beta z \rp = T_{m,\beta}(z) + R_{m,\beta}(z), \ \ \ T_{m,\beta}(z) := \sum_{n=0}^m \dfrac{\lp 2\pi \beta z \rp^n}{n!}.
\end{align*}
It is then easy to see that for
\begin{align*}
	H_{m,\beta}(z) := z^m T_{m,\beta}\lp 1/z \rp
\end{align*}
the previous heuristic approximations say that
\begin{align*}
	q_{f,a}(z) \approx N^{-(w+1)} H_{m,N}(z).
\end{align*}
Since the zeros of a polynomial are continuous functions of its coefficients, the zeros of $q_f(z)$ should behave like the zeros of $H_{m,N}(z)$. Now, the zeros of $T_{m,N}(z)$ go towards infinity as $m$ increases because the exponential function does not vanish on $\C$, and so the zeros of $H_{m,N}(z)$ must tend to zero and therefore should eventually lie inside the unit disk. The main technical tool we need for making this argument rigorous is Lemma \ref{L-function lemma}, which justifies the approximation $L(f,s) \approx N^{-s}$ in rigorous terms. The comparison of the zeros of $H_{m,N}(z)$ and $q_f(z)$ will then be done rigorously using Rouch\'{e}'s theorem.

We end this section with the required technical estimates for $L$-functions, leaving the study of the zeros of $H_{m,N}(z)$ until the following section. The method we use for proving the lemma is a modification of the key lemma in the work of Conrey, Farmer, and Imamoglu \cite{conreyfarmerimamoglu}.

\begin{lemma} \label{L-function lemma}
	Let $f(z) = q^N + \sum\limits_{n > N} c(n) q^n \in S_k$ be a nonzero normalized cusp form. Let $C_f$ be a constant such that $|c(n)| \leq C_f \sigma_0(n) n^{\frac{k-1}{2}}$ for all $n \geq 1$. Then we have for all $\sigma \geq \frac{3k}{4}$ that
	\begin{align*}
		\left| L\lp f, \sigma \rp - N^{-\sigma} \right| < \dfrac{4 C_f}{\lp N + 1 \rp^{k/4}}.
	\end{align*}
	Furthermore, for any integer $\sigma \geq \frac{k}{2}$ we have
	\begin{align*}
		\left| L\lp f, \sigma \rp \right| < 2 C_f \left[ 2 \sqrt{k} \log\lp 2k \rp - 2\sqrt{N-1} + 1 + 2^{k/2 + 1} e^{-\pi k} \right].
	\end{align*}
\end{lemma}

\begin{proof}
	Using the triangle inequality along with the bounds $|c(n)| \leq C_f \sigma_0(n) n^{\frac{k-1}{2}}$, $\sigma \geq \frac{3k}{4}$ and the trivial bound $\sigma_0(n) \leq 2 \sqrt{n}$, we have
	\begin{align*}
		\left| L\lp f, \sigma \rp - N^{-\sigma} \right| = \sum_{n \geq N+1} \dfrac{\left| c(n) \right|}{n^\sigma} &\leq C_f \sum_{n \geq N+1} \dfrac{\sigma_0(n)}{n^{\frac{3k}{4} - \frac{k}{2} + \frac 12}} \leq 2 C_f \sum_{n \geq N+1} \dfrac{1}{n^{k/4}}.
	\end{align*}
	From the valence formula we can see that $N \leq \frac{k}{12}$, and so we have $\frac{N+1}{k-4} \leq \frac 14$ for $k \geq 12$. Using this estimate along with elementary Riemann sums, we have
	\begin{align*}
		\sum_{n \geq N+1} \dfrac{1}{n^{k/4}} < \dfrac{1}{\lp N+1 \rp^{k/4}} + \int_{N+1}^\infty \dfrac{1}{x^{k/4}} dx = \dfrac{1 + \frac{4}{k-4} \lp N+1 \rp}{\lp N+1 \rp^{k/4}} < \dfrac{2}{\lp N+1 \rp^{k/4}}
	\end{align*}
	and therefore
	\begin{align*}
		\left| L\lp f, \sigma \rp - N^{-\sigma} \right| < \dfrac{4 C_f}{\lp N+1 \rp^{k/4}}
	\end{align*}
	This gives the first part of the lemma. To prove the second part, we first suppose that $\sigma \geq \frac{k}{2} + 1$ is an integer. We let $\zeta(s) := \sum_{n=1}^\infty n^{-s}$ be the classical Riemann zeta function. It is known classically using Dirichlet convolutions that $\zeta(s)^2 = \sum_{n=1}^\infty \sigma_0(n) n^{-s}$. Using the bounds assumed in the lemma, we have immediately that
	\begin{align} \label{Est 1}
		\left| L\lp f, \sigma \rp \right| \leq C_f \sum_{n = N}^\infty \dfrac{\sigma_0(n)}{n^{3/2}} < C_f \cdot \zeta\lp \frac 32 \rp^2.
	\end{align}
	We now consider the case $\sigma = \frac{k}{2}$. Using standard techniques of analytic number theory, we have
	\begin{align*}
		\Gamma\lp \dfrac{k}{2} \rp L\lp f, \dfrac{k}{2} \rp = 2 \sum_{n = N}^\infty \dfrac{c(n)}{n^{k/2}} \int_{2\pi n}^\infty e^{-x} x^{k/2} \dfrac{dx}{x}.
	\end{align*}
	Thus, we have
	\begin{align*}
		\left| L\lp f, \frac{k}{2} \rp \right| &\leq \dfrac{2}{\Gamma\lp k/2 \rp} \sum_{n = N}^\infty \dfrac{\left| c(n)\right|}{n^{k/2}} \int_{2\pi n}^\infty e^{-x} x^{k/2} \dfrac{dx}{x} \\ &\leq \dfrac{2 C_f}{\Gamma\lp k/2 \rp} \sum_{n=N}^\infty \dfrac{\sigma_0(n)}{\sqrt{n}} \int_{2\pi n}^\infty e^{-x} x^{k/2} \dfrac{dx}{x}.
	\end{align*}
	Since $N < k$, by completing the integral down to zero we have
	\begin{align*}
		\sum_{n = N}^{k} \dfrac{\sigma_0(n)}{\sqrt{n}} \int_{2\pi n}^\infty e^{-x} x^{k/2} \dfrac{dx}{x} < \Gamma\lp k/2 \rp \sum_{n = N}^{k} \dfrac{\sigma_0(n)}{\sqrt{n}}.
	\end{align*}
	By the argument in \cite{conreyfarmerimamoglu} we have
	\begin{align*}
		\sum_{n = 1}^{k} \dfrac{\sigma_0(n)}{\sqrt{n}} \leq 2 \sqrt{k} \log\lp 2 k \rp,
	\end{align*}
	and by using a simple induction argument along with $\sigma_0(n) \geq 1$ it may be shown that
	\begin{align*}
		\sum_{n = 1}^{N-1} \dfrac{\sigma_0(n)}{\sqrt{n}} > 2\sqrt{N-1} - 1.
	\end{align*}
	Thus, it follows that
	\begin{align*}
		\sum_{n = N}^{k} \dfrac{\sigma_0(n)}{\sqrt{n}} \int_{2\pi n} e^{-x} x^{k/2} \dfrac{dx}{x} < \Gamma\lp k/2 \rp \left[ 2\sqrt{k}\log\lp 2k \rp - 2 \sqrt{N-1} + 1 \right].
	\end{align*}
	We also have
	\begin{align*}
		\sum_{n = k+1}^\infty \dfrac{\sigma_0(n)}{\sqrt{n}} \int_{2\pi n}^\infty e^{-x} x^{k/2} \dfrac{dx}{x} &\leq \sum_{n = k+1}^\infty \dfrac{\sigma_0(n)}{\sqrt{n}} \int_{2\pi n}^\infty e^{-\pi n} e^{-x/2} x^{k/2} \dfrac{dx}{x} \\ &= \sum_{n = k+1}^\infty \dfrac{\sigma_0(n)}{\sqrt{n}} 2^{k/2} e^{-\pi n} \int_{\pi n}^\infty e^{-u} u^{k/2} \dfrac{du}{u} \\ &< 2^{k/2} \Gamma\lp k/2 \rp \sum_{n = k+1}^\infty \dfrac{\sigma_0(n)}{\sqrt{n}} e^{-\pi n} \\ &< 2 \Gamma\lp k/2 \rp \cdot 2^{k/2} e^{-\pi k}.
	\end{align*}
	We therefore have for $\sigma = \frac{k}{2}$ after some simplifying that
	\begin{align} \label{Est 2}
		\left| L\lp f, \sigma\rp \right| < 2 C_f \left[ 2 \sqrt{k} \log\lp 2k \rp - 2\sqrt{N-1} + 1 + 2^{k/2 + 1} e^{-\pi k} \right].
	\end{align}
	Using \eqref{Est 1}, we see that \eqref{Est 2} holds for all $\sigma \geq \frac{k}{2}$ whenever
	\begin{align*}
		\dfrac{1}{4} \zeta\lp 3/2 \rp^2 - \dfrac{1}{2} < \sqrt{k} \log\lp 2k \rp - \sqrt{N-1} + 2^{k/2} e^{-\pi k}.
	\end{align*}
	The left-hand side of this expression is $\approx 1.206$. Note that since $N < k$, it would be enough to show that
	\begin{align*}
		\dfrac{1}{4} \zeta\lp 3/2 \rp^2 - \dfrac{1}{2} < \sqrt{k} \log\lp 2k \rp - \sqrt{k} + 2^{k/2} e^{-\pi k}.
	\end{align*}
	The real valued function $g(x) := \sqrt{x} \log\lp 2x \rp -\sqrt{x} + 2^{x/2} e^{-\pi x}$ has $g(x) > \frac{1}{4} \zeta(3/2)^2 - \frac 12$ for all $x > 3$. In particular, since $f \in S_k$ is a nonzero cusp form of level 1 we have $k \geq 12$, and therefore \eqref{Est 2} applies for all $\sigma \geq \frac{k}{2}$.
\end{proof}

\subsection{Truncated exponential series}

This section gives a brief discussion of these zeros of $H_{m,N}(z)$ as $m$ becomes large. Recall the Taylor decomposition
\begin{align*}	
	\exp\lp 2\pi N z \rp = T_{m,N}(z) + R_{m,N}(z), \ \ \ T_{m,N}(z) := \sum_{n=0}^m \dfrac{\lp 2\pi N z \rp^n}{n!},
\end{align*}
and the related polynomial
\begin{align*}
	H_{m,N}(z) := z^m T_{m,N}\lp 1/z \rp.
\end{align*}
In this section, we study the zeros of $H_{m,N}(z)$ with a view towards showing that generically all zeros of $H_{m,N}(z)$ are inside the closed unit disk.

\begin{proposition} \label{H zeros theorem}
    Suppose $m,N \geq 1$ are integers such that $1 \leq N \leq \frac{1}{4} \log\lp m \rp$ and $m \geq 210$. Then
    \begin{align*}
        \left| T_{m,N}(z) \right| \geq e^{-2\pi N} - \dfrac{\lp 2\pi N \rp^{m+1}}{\lp m+1 \rp!} e^{2\pi N} > 0
    \end{align*}
    for every $|z| \leq 1$. Under the same conditions, $H_{m,N}(z)$ has all its zeros in the closed unit disk. Furthermore, $H_{m,1}(z)$ has all its zeros in the unit disk for $m \geq 20$.
\end{proposition}

\begin{remark}
    Our main results only use the case $N = 1$ and $m \geq 20$. However, the more general formulation of the proof is relevant for the discussion of the case $N \geq 2$ in Section \ref{Further directions}.
\end{remark}

\begin{proof}[Proof of Proposition \ref{H zeros theorem}]
    The claim regarding $H_{m,1}(z)$ is proven in \cite[Theorem 3.1]{elguindyraji}, so we proceed to the more general case. Let $z$ be a complex number with $|z| = r$ with $0 \leq r \leq 1$. Let $f(z) = \exp\lp 2\pi N z \rp$. Note that by Taylor's theorem, we have the bound
    \begin{align*}
        \left| R_{m,N}(z) \right| \leq \dfrac{\left| f^{(m+1)}(z_0) \right|}{\lp m+1 \rp!} |z|^{m+1}
    \end{align*}
    for some $0 \leq z_0 \leq r$. Since $f^{(m+1)}(z) = \lp 2\pi N \rp^{m+1} e^{2\pi N z}$, we therefore have
    \begin{align*}
        \left| R_{m,N}(z) \right| \leq \dfrac{\lp 2\pi N \rp^{m+1}}{\lp m+1 \rp!} r^{m+1} e^{2\pi r N} \leq \dfrac{\lp 2\pi/N \rp^{m+1}}{\lp m+1 \rp!} e^{2\pi/N}
    \end{align*}
    since the upper bound is clearly an increasing function of $r$ for $N > 0$. Thus, we have by the triangle inequality, the fact that $0 \leq r \leq 1$, and the definition of $T_{m,N}(z)$ that for all $|z| \leq 1$ we have
    \begin{align} \label{T/H lower bound}
        \left| T_{m,N}(z) \right| \geq \left| e^{2\pi N z} \right| - \left| R_{m,N}(z) \right| \geq e^{-2\pi N} - \dfrac{\lp 2\pi N \rp^{m+1}}{\lp m+1 \rp!} e^{2\pi N}.
    \end{align}
    Now, if we define the function
    \begin{align*}
        F_\alpha(m) := e^{-\alpha} - \dfrac{\alpha^{m+1}}{\lp m+1 \rp!} e^\alpha,
    \end{align*}
    it will be enough to show that under the hypotheses of the theorem, we have $F_{2\pi N}(m) > 0$. Now, with some elementary manipulation we for that any $\alpha, m > 0$ that $F_\alpha(m) > 0$ if and only if
    \begin{align*}
        2\alpha + \lp 2m+2 \rp \log(\alpha) < \log\lp m+1 \rp!.
    \end{align*}
    Now, for any $n \geq 1$ it is well known that
    \begin{align*}
        \log\lp n! \rp \geq n\log(n) - n + 1,
    \end{align*}
    and so to prove $F_\alpha(m) > 0$ it would be enough to show that
    \begin{align} \label{eq H 1}
         2\alpha + \lp 2m+2 \rp \log(\alpha) < \lp m+1 \rp \log\lp m+1 \rp - m.
    \end{align}
    Because $\alpha$ and $\log\lp \alpha \rp$ are both increasing functions of $\alpha$, it follows that if \eqref{eq H 1} holds for $\alpha = \alpha_0$ and $m \geq m_0$, then it also holds for $\alpha \leq \alpha_0$ and $m \geq m_0$. Setting $\alpha = 2\pi N$, \eqref{eq H 1} becomes
    \begin{align} \label{eq H 2}
        4\pi N + \lp 2m+2 \rp \log\lp 2\pi N \rp < \lp m+1 \rp \log\lp m+1 \rp - m,
    \end{align}
    and if we establish \eqref{eq H 2} for some $N = N_0 \geq 1$ and $m \geq m_0$, then it is also established for $1 \leq N \leq N_0$ and $m \geq m_0$. Now, a straightforward computation verifies \eqref{eq H 2} for $N = \frac{1}{4} \log\lp m \rp$ for all $m \geq 210$, and therefore for all $1 \leq N \leq \frac{1}{4} \log\lp m \rp$ and all $m \geq 210$, \eqref{eq H 2} holds, and by \eqref{T/H lower bound} we obtain
    \begin{align*}
        \left| T_{m,N}(z) \right| \geq \left| e^{2\pi N z} \right| - \left| R_{m,N}(z) \right| \geq e^{-2\pi N} - \dfrac{\lp 2\pi N \rp^{m+1}}{\lp m+1 \rp!} e^{2\pi N} > 0.
    \end{align*}
    This proves the first part of the proposition. The statement about the zeros of $H_{m,N}(z)$ follows from the first part along with the identity $H_{m,N}(z) = z^m T_{m,N}(1/z)$.
\end{proof}

\begin{remark}
    One can see easily from the definition of $H_{m,N}(z)$ that if the hypotheses of Proposition \ref{H zeros theorem} hold and if $|z| = 1$, then
    \begin{align} \label{H lower bound}
        \left| H_{m,N}(z) \right| \geq e^{-2\pi N} - \dfrac{\lp 2\pi N \rp^{m+1}}{\lp m+1 \rp!} e^{2\pi N} > 0.
    \end{align}
\end{remark}

\section{Proofs of main results} \label{Main Proofs}

\subsection{Proof of Theorem \ref{Main Theorem}}

Let $f \in S_k$ be a nonzero cusp form with an order of vanishing $N \geq 1$ at infinity. Throughout the proof we use the notation $m = \frac{w}{2} = \frac{k}{2} - 1$. We begin with calculations that compare the polynomials $q_f(z)$ and $N^{-(w+1)} H_{m,N}(z)$. By Lemma \ref{L-function lemma}, we have for all $|z| = 1$ that
\begin{align*}
	\big| q_f(z) &- N^{-(w+1)} H_{m,N}(z) \big| \\ &\leq \sum_{n=0}^{m-1} \left| L\lp f, w-n+1 \rp - N^{-(w-n+1)} \right| \dfrac{\lp 2\pi \rp^n}{n!} + \lp \left|L\lp f, w-m+1 \rp\right| + N^{-(w-m+1)} \rp \dfrac{1}{m!} \\ &\leq \sum_{n=0}^{\lfloor \frac{k}{4} \rfloor - 1} \left| L\lp f, w-n+1 \rp - N^{-(w-n+1)} \right| \dfrac{\lp 2\pi \rp^n}{n!} \\ & \hspace{.2in} + \sum_{n = \lfloor \frac{k}{4} \rfloor}^m \lp \left| L\lp f, w-n+1 \rp \right| + N^{-(w-n+1)} \rp \dfrac{\lp 2\pi \rp^n}{n!} \\ &< \sum_{n = 0}^{\lfloor \frac{k}{4} \rfloor - 1} E_1\lp f \rp \dfrac{\lp 2\pi \rp^n}{n!} + \sum_{n = \lfloor \frac{k}{4} \rfloor}^m \lp E_2\lp f \rp + N^{-(w-n+1)} \rp \dfrac{\lp 2\pi \rp^n}{n!},
\end{align*}
where we define
\begin{align*}
    E_1\lp f \rp := \dfrac{4 C_f}{\lp N + 1 \rp^{k/4}}
\end{align*}
and
\begin{align*}
    E_2\lp f \rp := 2 C_f \left[ 2 \sqrt{k} \log\lp 2k \rp - 2\sqrt{N-1} + 1 + 2^{k/2 + 1} e^{-\pi k} \right].
\end{align*}
The first summand in the above upper bound may be bounded by
\begin{align*}
    \sum_{n = 0}^{\lfloor \frac{k}{4} \rfloor - 1} E_1\lp f \rp \dfrac{\lp 2\pi \rp^n}{n!} = \dfrac{4C_f}{\lp N+1 \rp^{k/4}} \sum_{n=0}^{\lfloor \frac{k}{4} \rfloor - 1} \dfrac{\lp 2\pi \rp^n}{n!} < \dfrac{4 e^{2\pi} C_f}{\lp N+1 \rp^{k/4}} =: C_f \alpha_{k,N}.
\end{align*}
The second summand can be broken down into two pieces, which we bound separately using Taylor's theorem. Then we have using $N \geq 1$, $k \geq 12$, and $\sqrt{k} \log\lp 2k \rp < k$ for $k \geq 12$ that
\begin{align*} 
      \sum_{n = \lfloor \frac{k}{4} \rfloor}^m E_2\lp f \rp \dfrac{\lp 2\pi \rp^n}{n!} &< 2 C_f \left[ 2 \sqrt{k} \log\lp 2k \rp - 2\sqrt{N-1} + 1 + 2^{k/2 + 1} e^{-\pi k} \right] \sum_{n = \lfloor \frac{k}{4} \rfloor}^\infty \dfrac{\lp 2\pi \rp^n}{n!} \\ &< \dfrac{4 e^{2\pi} C_f \lp 2\pi \rp^{\lfloor \frac{k}{4} \rfloor} \left[ \sqrt{k} \log\lp 2k \rp + 1 \right]}{\lp \lfloor \frac{k}{4} \rfloor \rp!} =: C_f \beta_k
\end{align*}
by applying Taylor's theorem, and likewise
\begin{align*}
     \sum_{n = \lfloor \frac{k}{4} \rfloor}^m N^{-(w-n+1)} \dfrac{\lp 2\pi \rp^n}{n!} < N^{-(w+1)} \sum_{n = \lfloor \frac{k}{4} \rfloor}^\infty \dfrac{\lp 2\pi N \rp^n}{n!} < \dfrac{e^{2\pi N} \lp 2\pi N \rp^{\lfloor \frac k4 \rfloor}}{N^{w+1} \lp \lfloor \frac k4 \rfloor \rp!} =: \gamma_{k,N}.
\end{align*}
Therefore, we have
\begin{align} \label{Eq 1}
    \big| q_f(z) &- N^{-(w+1)} H_{m,N}(z) \big| < C_f \alpha_{k,N} + C_f \beta_k + \gamma_{k,N}.
\end{align}
Now, recalling \eqref{H lower bound}, we have if $m \geq 210$ and $1 \leq N \leq \frac 14 \log\lp m \rp$ that
\begin{align*}
    \left| N^{-(w+1)} H_{m,N}(z) \right| \geq \delta_{k,N} > 0,
\end{align*}
where
\begin{align*}
     \delta_{k,N} := N^{-(w+1)} \lp e^{-2\pi N} - \dfrac{\lp 2\pi N \rp^{k/2}}{\lp k/2 \rp!} e^{2\pi N} \rp.
\end{align*}
Thus, if we assume that $m \geq 210$, $1 \leq N \leq \frac{1}{4} \log\lp m \rp$, and the inequality
\begin{align} \label{Key inequality}
    C_f \alpha_{k,N} + C_f \beta_k + \gamma_{k,N} < \delta_{k,N},
\end{align}
can combine Proposition \ref{H zeros theorem} with equations \eqref{Eq 1} and \eqref{Key inequality} to obtain
\begin{align*}
    \left| q_f(z) - N^{-(w+1)} H_{m,N}(z) \right| < \left| N^{-(w+1)} H_{m,N}(z) \right|.
\end{align*}
Thus, if we assume \eqref{Key inequality}, the hypotheses of Rouch\'{e}'s theorem are satisfied and we see that $q_f(z)$ and $H_{m,N}(z)$ have the same number of zeros in the unit disk. We can then invoke Proposition \ref{H zeros theorem} to determine how many zeros $q_f(z)$ has inside the unit disk. In particular, if $m,N$ are such that $H_{m,N}(z)$ has all its zeros inside the unit disk, then since $q_f(z)$ and $H_{m,N}(z)$ have the same degree, $q_f(z)$ also has all its zeros in the unit disk, and then by Proposition \ref{q zeros} we have that $r_f(z)$ is unimodular. We have thus shown that the truth of the inequality \eqref{Key inequality} is sufficient to prove that $r_f(z)$ is unimodular, which completes the proof of Theorem \ref{Main Theorem}.

\subsection{Proof of Theorem \ref{U-bound theorem}}

We now prove a very general corollary of Theorem \ref{Main Theorem} which is useful for studying certain families of cusp forms. Philosophically, this theorem tells us that if $f \in S_k$ is a linear combination of eigenforms $f = c_1 f_1 + \dots c_r f_r$, then $r_f(z)$ will satisfy the Riemann hypothesis provided that the $c_j$'s do not cancel each other too much.

\begin{theorem} \label{U-bound theorem}
 Let $k \geq 180$ be an even integer and let $f = c_1 f_1 + \dots + c_r f_r$ be a weight $k$ cusp form, where $f_j$ denote the eigenforms of $S_k$ and $c_j \in \C$, not all zero. Define the function $U(k)$ by
 \begin{align*}
    U(k) := \dfrac{0.001865}{4 e^{2\pi}} 2^{k/4},
\end{align*}
and suppose the coefficients $c_i$ satisfy the inequality
\begin{align} \label{ineq}
    \sum_{j=1}^r \left| c_j \right| \leq U(k) \cdot \left| \sum_{j=1}^r c_j \right|.
\end{align}
Then the period polynomial $r_f(z)$ satisfies the Riemann hypothesis. In particular, if $c_j \in \R$ for all $j$ and we define $C^\pm := \pm \sum_{\pm c_j > 0} c_j$, then if either of the inequalities
\begin{align*}
    C^\pm \leq \dfrac{U(k) - 1}{U(k) + 1} C^\mp
\end{align*}
hold, then $r_f(z)$ satisfies the Riemann hypothesis.
\end{theorem}

\begin{proof}
    Let $f = c_1 f_1 + \dots + c_r f_r$ as stated in the theorem. We can assume without loss of generality that $C := \sum_{j=1}^r c_j \not = 0$, since otherwise \eqref{ineq} fails. Then $f = Cq + O\lp q^2 \rp$, we can without loss of generality normalize to $f^* = \frac{1}{C} f = q + O\lp q^2 \rp$. Now, if $f = \sum_{n \geq 1} a(n) q^n$, Deligne's theorem implies that $\left| a(n) \right| \leq C^* \sigma_0(n) n^{\frac{k-1}{2}}$, where $C^* := \sum_{j=1}^r \left| c_j \right|$. In the notation above, we have $C_{f^*} \leq \frac{C^*}{C}$. Observe that the main condition required on the $c_j$ in the statement of the theorem can be rephrased as $\frac{C^*}{C} \leq U(k)$. Thus, it will suffice to prove that if $g \in S_k$ is a normalized cusp form, the Riemann hypothesis for $r_g(z)$ holds whenever $C_g \leq U(k)$. 
    
    We prove this claim by carefully computing bounds on each term in \eqref{N=1 eqn v1}. It is straightforward to compute the bounds $\delta_{k,1} > 0.001867$ and $\gamma_{k,1} < 10^{-10}$ for all $k \geq 150$ from these definitions, and therefore we may prove \eqref{N=1 eqn v1} by proving the simpler inequality
    \begin{align} \label{N=1 eqn v2}
        C_g \alpha_{k,1} + C_g \beta_k < 0.001866.
    \end{align}
    Now, using the inequality $\frac{\sqrt{k+2} \log\lp 2k+4 \rp + 1}{\sqrt{k} \log\lp 2k \rp + 1} < 1.01$, valid for $k \geq 150$, we may derive in a straightforward manner the inequalities
    \begin{align*}
        \beta_{k+2} < \begin{cases} 1.01 \beta_k & \text{if } k \equiv 0 \pmod{4}, \\ \dfrac{1.01 \times 2\pi}{\frac{k}{4} + 1} \beta_k & \text{if } k \equiv 2 \pmod{4}. \end{cases}
    \end{align*}
    Noting that $150 \equiv 2 \pmod{4}$, $\frac{k}{4} - 1 > 36$ for $k \geq 150$, and $\beta_{150} < 10^{-8}$, we see that for every integer $\ell \geq 0$ and $k = 150 + 2\ell$, we have
    \begin{align*}
        \beta_k < \lp 1.01 \rp^\ell \lp \dfrac{2\pi}{36} \rp^{\lceil \frac{\ell}{2} \rceil} \beta_{150} < \lp \dfrac{3}{7} \rp^\ell 10^{-8}.
    \end{align*}
    This condition implies that $\beta_k < \frac{1}{C_g} 10^{-6}$ whenever $C_g < 100 \lp \frac 73 \rp^{\frac{k-150}{2}} =: U^*(k)$, and therefore if we assume this condition on $C_g$ then the inequality \eqref{N=1 eqn v2} would follow from
    \begin{align*}
        \alpha_{k,1} < \dfrac{0.001865}{C_g},
    \end{align*}
    which may be quickly simplified to
    \begin{align} \label{N=1 eqn v3}
        C_g < \dfrac{0.001865}{4 e^{2\pi}} 2^{k/4} = U(k).
    \end{align}
    We therefore conclude from Theorem \ref{Main Theorem} that if $C_g < \min\lp U^*(k), U(k) \rp$, the period polynomial of $g$ is unimodular. It is easy to see that $U(k) < U^*(k)$ for $k \geq 180$, so in fact we need only show that $C_g < U(k)$ for $k \geq 180$, which completes the proof.
\end{proof}

\subsection{Proof of Theorems \ref{Positive sums theorem} and \ref{Constellations}}

We first prove Theorem \ref{Positive sums theorem}. Note that because $c_j \geq 0$ for all $j$, we can without loss of generality specialize to the case $N=1$ in Theorem \ref{Main Theorem}. Thus, by Proposition \ref{H zeros theorem}, the assumptions $k \geq 42$ and
\begin{align} \label{N=1 eqn v1}
    C_f \alpha_{k,1} + C_f \beta_k + \gamma_{k,1} < \delta_{k,1}
\end{align}
are sufficient to prove the unimodularity of the period polynomial of $f$. We first prove \eqref{N=1 eqn v1}. As in previous proofs, note that if $C = c_1 + \dots + c_r$ and $f^*(z) := \frac{1}{C} f(z) = q + O\lp q^2 \rp$, we need only show that $r_{f^*}(z)$ satisfies the Riemann hypothesis. Now, since $f$ is a linear combination of $C$ eigenforms, the Deligne bound tells us that $C_f \leq C$, and so $C_{f^*} \leq 1$. Therefore, to prove that $f$ has a unimodular period polynomial we need only prove $\alpha_{k,1} + \beta_k + \gamma_{k,1} < \delta_{k,1}$. This is clearly true asymptotically, since  $\alpha_{k,1}, \beta_{k,1}, \gamma_{k,1} \to 0$ as $k \to \infty$ and $\delta_{k,1} > 0$ and is an increasing function of $k$; indeed, this inequality holds for $k \geq 120$, and so Theorem \ref{Positive sums theorem} follows.

We now prove Theorem \ref{Constellations}. For this result, it is more convenient to use Theorem \ref{U-bound theorem}. If $f \in \mathcal F\lp \mathbf{c}, \mathbf{a} \rp$, i.e. $f = c_1 f_{i_1} + \dots + c_m f_{i_m} - a_1 f_{j_1} - \dots - a_n f_{j_n}$, we have in the notation of Theorem \ref{U-bound theorem} that $C^+ = \sum_{i=1}^m c_i$, $C^- = \sum_{j=1}^n a_j$, and $C^+ > C^- \geq 0$. By Theorem \ref{U-bound theorem} it would suffice to prove that
\begin{align*}
    C^- \leq \dfrac{U(k) - 1}{U(k)+1} C^+
\end{align*}
holds for $k \gg 0$. This is clear from the fact that $C^+ > C^-$ and $\frac{U(k) - 1}{U(k)+1} \to 1$ as $k \to \infty$.

\section{Probabilistic interpretations} \label{Probability}

In this section, we discuss some probabilistic considerations coming from Theorem \ref{Positive sums theorem}. Before stating these results, we need to set up notation. For integers $X \geq 1$ and $k \geq 12$ even, define the sets
\begin{align*}
        \mathcal S_{k,X} &:= \{ f = c_1 f_1 + \dots + c_r f_r : -X \leq c_j \leq X, c_j \in \Z \}, \\
        \mathcal S_{k,X}^{\mathrm{RH}} &:= \{ f \in \mathcal S_{k,X} : r_f\lp z \rp \text{ satisfies the Riemann hypothesis} \}.
\end{align*}
We also consider the function
\begin{align*}
    \mathcal P_k\lp X \rp := \dfrac{\left| \mathcal S_{k,X}^{\mathrm{RH}} \right|}{\left| \mathcal S_{k,X} \right|},
\end{align*}
which can be interpreted as the probability that the period polynomial of a randomly selected cusp form $f \in \mathcal S_{k,X}$ satisfies the Riemann hypothesis. The following corollary follows immediately from Theorem \ref{Positive sums theorem} by considering the probability that either $c_j \leq 0$ or $c_j \geq 0$ for all $j$.

\begin{corollary} \label{Probability corollary}
    Let $X \geq 1$ be an integer and let $k \geq 180$. If $r = \dim\lp S_k \rp$, then
    \begin{align*}
        \mathcal P_k\lp X \rp \geq \dfrac{1}{2^{r-1}}.
    \end{align*}
\end{corollary}

\begin{remark}
    A similar statement could be proved allowing $X$ and $c_j$ to take on real values in terms of lower densities. For simplicity, we do not discuss this in detail.
\end{remark}

Conceptually, we interpret Corollary \ref{Probability corollary} as a statement that for large weights $k$, a randomly selected cusp form has a positive probability of satisfying the Riemann hypothesis. This raises many interesting questions about the nature of the values $\mathcal P_k\lp X \rp$, particularly about limiting values. Using results in this paper, we are able to prove the following limiting result as $k \to \infty$.

\begin{corollary} \label{Limit probability corolllary}
    Let $X \geq 1$ be an integer. Then we have
    \begin{align*}
        \lim\limits_{k \to \infty} \mathcal P_k\lp X \rp = 1.
    \end{align*}
\end{corollary}

\begin{proof}
    We use the notation of Theorem \ref{U-bound theorem}, we automatically have $f \in \mathcal S_{k,X}^{\mathrm{RH}}$ if $C^+ = 0$ or $C^- = 0$, and so we only consider the case $C^\pm \not = 0$. Suppose for the moment that $f$ has only a simple pole at $i\infty$, so that $C^+ \not = C^-$. Without loss of generality, assume $C^+ > C^-$. By Theorem \ref{U-bound theorem}, $f \in \mathcal S_{k,X}^{\mathrm{RH}}$ if
    \begin{align*}
        C^- \leq \dfrac{U(k) - 1}{U(k) + 1} C^+.
    \end{align*}
    This holds for $k \gg 0$, and so every $f \in \mathcal S_{k,X}$ with a simple pole at $i\infty$ lies in $\mathcal S_{k,X}^{\mathrm{RH}}$ if $k$ is sufficiently large. It thus only remains to show that the condition that $f$ has a simple pole at $i\infty$ has probability 1 as $k \to \infty$. Since $f = c_1 f_1 + \dots + c_r f_r$ has a simple pole at infinity if and only if $c_1 + \dots + c_r \not = 0$, the result follows immediately from the fact that $r \to \infty$ as $k \to \infty$.
\end{proof}

\section{Further directions} \label{Further directions}

\subsection{Obstructions to unimodularity}

Although Theorem \ref{Main Theorem} gives a sufficient condition for unimodularity that in principle applies to any cusp form, it is not a necessary condition. In practice, Theorem \ref{Main Theorem} does not work well when $k$ is small or $N \geq 2$. Indeed, observe that even though we prove a version of Proposition \ref{H zeros theorem} which works quite well for $N \geq 2$, the asymptotic properties of \eqref{Key inequality 0} are very different for different $N$; indeed, if $C_f$ is bounded and $N \geq 1$ is fixed, then \eqref{Key inequality 0} is true for $k \gg 0$ if and only if $N = 1$. One way to try to rectify the situation would be to look only at $f$ whose value of $C_f$ is very small. This, however, also does not work. To see this, consider a cusp form $f = q^N + \sum_{n \geq N+1} c(n) q^n$ of weight $k$ and $N \geq 2$. When the assumption $\left| c(n) \right| \leq C_f \sigma_0(n) n^{\frac{k-1}{2}}$ is applied to the case $c\lp N \rp = 1$, we have a lower bound $C_f \geq \lp \sigma_0(N) N^{\frac{k-1}{2}} \rp^{-1}$. Now, if we fix $N \geq 2$, in order for \eqref{Key inequality 0} to hold as $k \to \infty$, we require that $4 e^{2\pi} C_f \lp N+1 \rp^{-k/4}$ decay to zero faster than $N^{-(k-1)}$. This would require $C_f = o\lp N^{- (3k/4 - 1)} \rp$, which the previous lower bound on $C_f$ prevents. If one allows $N$ to increase as a function of $k$, the situation becomes even worse than this, as both the first and second terms on the left-hand side of \eqref{Key inequality 0} no longer decay sufficiently quickly as $k \to \infty$. We also note in passing that when the dimension of $S_k$ is larger than 1, normalized cusp forms $f \in S_k$ can have arbitrarily large values for $C_f$; this is a consequence of the existence of the Sturm basis for $S_k$. Thus, it does not appear that Theorem \ref{Main Theorem} can be used on its own to determine precisely which elements of $f \in S_k$ have period polynomials satisfying the Riemann hypothesis.

Therefore, if we wish to study the Riemann hypothesis for period polynomials when $f$ has an order of vanishing $N \geq 2$ or in families of cusp forms with variable orders of vanishing, either Theorem \ref{Main Theorem} must be improved, or a new method must be developed. As there does not seem to be any better way to choose the polynomials $q_f(z)$, the most plausible way to improve Theorem \ref{Main Theorem} would be to improve the estimates in Lemma \ref{L-function lemma}. It is in this lemma where we can see the reason that Theorem \ref{Main Theorem} is only useful for $N = 1$; this is because for a normalized cusp form $f$, the approximation $L\lp f, s \rp \approx N^{-s}$ for $s \gg 0$ is only a good approximation when $N = 1$. This is because in a Dirichlet series $\sum_{n=1}^\infty a(n) n^{-s}$, only the term $n=1$ is independent of $s$, and all other terms decay to 0 as $s \to \infty$. Thus, if $N \geq 2$, the ``main term" $N^{-s}$ and the ``error term" $L\lp f,s \rp - N^{-s}$ both decay to 0 with a similar speed as $s \to \infty$. Therefore, in order to prove stronger results on the zeros of period polynomials, one should produce non-trivial improvements of Lemma \ref{L-function lemma}. In particular, only the first part of Lemma \ref{L-function lemma} needs to be stronger if one wishes to study the case where $N \geq 2$ is fixed, and both need to be improved in order to address families of cusp forms where $N \to \infty$ as $k \to \infty$.

\subsection{Bounding the constant $C_f$}

For any arbitrary form $f = \sum_{n \geq 1} c(n) q^n$, the problem of giving explicit estimates for the coefficients of $f$ is very difficult. If $f$ is a normalized eigenform, the Deligne bound says that $\left| c(n) \right| \leq \sigma_0(n) n^{\frac{k-1}{2}}$. For a general cusp form $f$, the fact that the spaces $S_k$ always have a basis of eigenforms implies that there is some constant $C_f$ such that $\left| c(n) \right| \leq C_f \sigma_0(n) n^{\frac{k-1}{2}}$. In particular, if $f_1, \dots, f_r$ is the eigenbasis of $S_k$ and if $f = c_1 f_1 + \dots + c_r f_r$ for $c_j \in \C$, then we have $C_f \leq \sum_{j=1}^r \left| c_j \right|$. Although upper bounds are difficult, it is quite straightforward to show that among normalized cusp forms $f$, $C_f$ can take on arbitrarily large values as long as the space $S_k$ has a dimension greater than 1. We are therefore only concerned with computing upper bounds on $C_f$ which depend on some properties of $f$.

There are various approaches in the literature for obtaining bounds on $C_f$, we refer the reader to \cite{spyenirce} for a list of some references. Although many references contain only inexplicit estimates, we mention two results which give bounds that are totally explicit. One such result comes from the work of Schulze-Pilot and Yenirce \cite{spyenirce}, who compute bounds for the Fourier coefficients of newforms using Petersson products. In particular, they show that if
\begin{align*}
		\left|c(n)\right| \leq 2 \sqrt{\pi} e^{2\pi} \sqrt{\langle f, f \rangle} \sigma_0(n) n^{\frac{k-1}{2}},
\end{align*}
where $\sigma_0(n)$ is the number of divisors of $n$ and $\langle f, g \rangle$ is the Petersson inner product of $f,g \in S_k$ of level 1. We note briefly that the full result of \cite{spyenirce} is stated for cusp forms on $\Gamma_0(N)$ with Nebentypus. Another quite different approach can be found in the work of Jenkins and Rouse \cite{jenkinsrouse}, who bound the Fourier coefficients of a cusp form by its first $r$ coefficients. In particular, they prove that
\begin{align*}
    \left| c(n) \right| \leq \sqrt{\log(k)} \lp 11 \sqrt{\sum_{m=1}^r \dfrac{|c(m)|^2}{m^{k-1}}} + \dfrac{e^{18.72} \lp 41.41 \rp^{k/2}}{k^{(k-1)/2}} \left| \sum_{m=1}^r c(m) e^{-7.288m} \right| \rp \sigma_0(n) n^{\frac{k-1}{2}}.
\end{align*}
We note that although the formula of \cite{jenkinsrouse} is only written for $\mathrm{SL}_2\lp \Z \rp$, it has the benefit that the role played by $N$ in bounding $C_f$ is completely explicit. Alongside a suitable improvement of Theorem \ref{Main Theorem}, this could be used to study cusp forms with a large order of vanishing at infinity.

\subsection{Zeta polynomials}

We describe briefly the application of the unimodularity of period polynomials to the construction of zeta polynomials in the sense of Manin \cite{manin_zetapoly} and Ono, Rolen, and Sprung \cite{onorolensprung}.
\begin{definition}
    A {\it zeta polynomial} is a polynomial $Z(s)$ which satisfies a functional equation $Z(s) = \pm Z(1-s)$, has all its zeros on the line $\mathrm{Re}\lp \rho \rp = \frac 12$, and have an arithmetic-geometric interpretation.
\end{definition}
In \cite{manin_zetapoly}, Manin constructs zeta polynomials from the odd parts of period polynomials, and Ono, Rolen, and Sprung construct zeta polynomials from the full period polynomials of eigenforms. The key ingredients in these proofs are the Riemann hypothesis for the (odd or full) period polynomials along with a criterion of Rodriguez-Villegas \cite{rodrigovillegas} that transforms unimodular polynomials into zeta polynomials. To explain the relevant results, let $U \in \C[x]$ be a polynomial of degree $d$ with $U(1) \not = 0$, and let the {\it Rodriguez-Villegas transform} $Z(s)$ of $U(x)$ be the unique polynomial such that
\begin{align*}
    \dfrac{U(x)}{\lp 1-x \rp^{d+1}} = \sum_{n \geq 0} Z\lp -n \rp x^n.
\end{align*}
In \cite[Theorem 2.1]{onorolensprung}, it is shown that if $U(x)$ is unimodular with real coefficients, then $Z(s)$ has all its zeros on $\mathrm{Re}\lp \rho \rp = \frac 12$ and satisfies the functional equation $Z(s) = (-1)^d Z(1-s)$. Using the above criterion, certain polynomials $Z_f(s)$ defined in \cite[Equation 1.3]{onorolensprung} are shown to be zeta polynomials. Since every step in their method is valid for more general cusp forms, we see that any cusp for $f$ for which $r_f(z)$ satisfies the Riemann hypothesis can be used to construct a zeta polynomial.

\subsection{Conjectures and open problems}

After \cite{conreyfarmerimamoglu, elguindyraji} established the Riemann hypothesis for odd and full period polynomials for $\mathrm{SL}_2\lp \Z \rp$ modular forms, many related directions have developed in the literature. In this final section, we discuss a variety of results, conjectures, and open problems which naturally flow from known results and our work.

\subsubsection{Other modular objects}

One of the most natural directions in which to generalize the results of \cite{conreyfarmerimamoglu, elguindyraji} is to cusp forms with level $N > 1$ and Nebentypus $\chi$. This has been accomplished for newforms $f \in S_k\lp \Gamma_0(N), \chi \rp$ with even weight $k \geq 4$ and trivial Nebentypus by Jin, Ma, Ono, and Soundararajan \cite{jinmaonosound} and for almost all newforms of any weight $k \geq 3$ and any Nebentypus by Liu, Park, and Song \cite{liuparksong}. It would be interesting to prove generalizations of the results in this paper to the setting of the spaces $S_k\lp \Gamma_0(N), \chi \rp$. These papers also prove equidistribution results for zeros of eigenforms as $k$ or $N$ grow, and it would be natural to ask similar distribution questions about the distribution of zeros in more general settings.

In light of \cite{conreyfarmerimamoglu}, the odd parts of the period polynomials of general cusp forms would also be interesting to investigate, and we expect that these should also satisfy an appropriate Riemann hypothesis for large families of cusp forms.

\subsubsection{Special families of cusp forms}

One reason why previous literature has focused on the period polynomials of eigenforms and newforms is that they are very natural and beautiful examples of cusp forms. The Fourier coefficients of a normalized newform $f(z) = \sum_{n \geq 1} c(n) q^n \in S_k\lp \Gamma_0(N), \chi \rp$ are multiplicative and the sequences $\{ c(p^n) \}_{n \geq 0}$ satisfy a two-term linear recurrence relation. These properties have many very beautiful implications, one of which is that the $L$-function $L\lp f,s \rp$ of a newform has an Euler product, which the authors of \cite{jinmaonosound} use to estimate ratios of critical $L$-values. It would be natural to study other special families of cusp forms to determine whether their period polynomials (resp. odd period polynomials) satisfy the Riemann hypothesis; one might consider, for example, CM forms or powers of eigenforms. For example, although Theorem \ref{Main Theorem} is not sufficiently refined to treat period polynomials of $\Delta^a$ for $a \gg 0$, explicit computations of the period polynomials of $\Delta^a$ for small $a$ suggests that very often these satisfy the Riemann hypothesis. It would also be natural to ask whether modular forms with maximal orders of vanishing at infinity (for their respective spaces $S_k$) have unimodular period polynomials. We leave these questions open for future investigation.

\subsubsection{Probabilities for fixed weights}

Corollary \ref{Limit probability corolllary} tells us that for $X \geq 1$ an integer, $\mathcal P_k\lp X \rp \to 1$ as $k \to \infty$, that is, that asymptotically 100\% of cusp forms which are ``linear combinations of a small number of eigenforms" satisfy the Riemann hypothesis. Questions about the behaviour of $\mathcal P_k\lp X \rp$ for small values of $k$ or when $X \to \infty$ are also natural. However, our methods do not seem to apply immediately to prove such results. The most we can prove is that $\liminf\limits_{X \to \infty} \mathcal P_k\lp X \rp > 0$ for $k \gg 0$, which is a consequence of Corollary \ref{Probability corollary}. We pose two questions related to $\mathcal P_k\lp X \rp$ which would make for interesting directions for future study.

\begin{question}
    For fixed integers $k$, do the limits
    \begin{align*}
        \lim_{X \to \infty} \mathcal P_k\lp X \rp
    \end{align*}
    exist? If so, are they equal to 1?
\end{question}

\begin{question}
    Let $12 \leq k < 120$ be an integer and $f = c_1 f_1 + \dots + c_r f_r \in S_k$, where $c_j \geq 0$ and $f_1, \dots, f_r$ is the normalized eigenbasis of $S_k$. Does the period polynomial $r_f(z)$ satisfy the Riemann hypothesis?
\end{question}

Not much can be shown regarding these questions using our method. For the first question, lower bounds on the $\liminf_{X \to \infty} \mathcal P_k\lp X \rp$ can be derived for $k \gg 0$ from Corollary \ref{Limit probability corolllary}, but these lower bounds are probably not optimal and do not imply that these limits even converge. For the second question, the answer is affirmative for any $k$ for which $S_k$ is one-dimensional by \cite{elguindyraji}, but nothing else seems to be known.

\end{document}